\documentclass[11 pt,reqno]{amsart}
\usepackage{amsmath}
\usepackage{amssymb}
\usepackage{amsthm}
\usepackage{amsfonts}
\usepackage{graphicx}
\usepackage{graphics}
\usepackage{epsfig}

\usepackage{color}

 \newtheorem{thm}{Theorem}[section]
\newtheorem*{Mth}{Theorem}
 
 \newtheorem{lem}[thm]{Lemma}
 \newtheorem{prop}[thm]{Proposition}
 \theoremstyle{definition}
 \newtheorem{defn}[thm]{Definition}

 \theoremstyle{remark}
 
 \newtheorem{ex}{Example}

\numberwithin{equation}{section} \numberwithin{figure}{section}

\newcommand{\CC}{{\mathbb C}}
\newcommand{\DD}{{\mathbb D}}
\newcommand{\EE}{\mathbb{E}}
\newcommand{\TT}{{\mathbb T}}

\newcommand{\RR}{{\mathbb R}}
\newcommand{\NN}{{\mathbb N}}

\newcommand{\ZC}{\mathcal{Z}}

\newcommand{\Dal}{\mathfrak{D}_{\alpha}}

\newcommand{\alphacap}{\text{cap}_\alpha}

\newcommand{\McC}{\raise.5ex\hbox{c}}
\newcommand{\ip}[2]{\langle #1, #2 \rangle}

\DeclareMathOperator{\dist}{dist}

\newcommand{\al}{\alpha}

\renewcommand{\phi}{\varphi}

\begin{document}
\bibliographystyle{alpha}

\title[Cyclic polynomials in two variables]{Cyclic polynomials in two variables}
\author[B\'en\'eteau]{Catherine B\'en\'eteau}
\address{Department of Mathematics, University of South Florida, 4202 E. Fowler Avenue, Tampa, FL 33620-5700, USA.}
\email{cbenetea@usf.edu}
\author[Knese]{Greg Knese}
\address{Washington University in St. Louis, Department of Mathematics, One Brookings Drive, Campus Box 1146, St. Louis, MO 63130-4899, USA.}
\email{geknese@math.wustl.edu}
\author[Kosi\'nski]{\L ukasz Kosi\'nski}

\address{Institute of Mathematics, Jagiellonian University, \L
  ojasiewicza 6, 30-348 Krak\'ow, Poland.}
\email{lukasz.kosinski@gazeta.pl} \author[Liaw]{Constanze Liaw}
\address{Department of Mathematics, Baylor University, One Bear Place
  \#97328, Waco, TX 76798-7328, USA.}
\email{Constanze$\underline{\,\,\,}$Liaw@baylor.edu}
\author[Seco]{Daniel Seco} \address{Mathematics Institute, Zeeman
  Building, University of Warwick, Coventry CV4 7AL, UK.}
\email{D.Seco@warwick.ac.uk} \author[Sola]{Alan Sola} \address{Centre
  for Mathematical Sciences, University of Cambridge, Wilberforce
  Road, Cambridge CB3 0WB, UK.}  \email{a.sola@statslab.cam.ac.uk}
\thanks{Knese is supported by NSF grant DMS-1363239. Kosi\'nski is
  supported by NCN grant 2011/03/B/ST1/04758. Liaw is partially
  supported by the NSF grant DMS-1261687. Seco is supported by ERC
  Grant 2011-ADG-20110209 from EU programme FP2007-2013 and MEC/MICINN
  Project MTM2011-24606.  Sola acknowledges support from the EPSRC
  under grant EP/103372X/1. } \date{\today}

\begin{abstract}
  We give a complete characterization of polynomials in two complex
  variables that are cyclic with respect to the coordinate shifts
  acting on Dirichlet-type spaces in the bidisk, which include the
  Hardy space and the Dirichlet space of the bidisk.  The cyclicity of
  a polynomial depends on both the size and nature of the zero set of
  the polynomial on the distinguished boundary.  The techniques in the
  proof come from real analytic function theory, determinantal
  representations for stable polynomials, and harmonic analysis on
  curves.
\end{abstract}

\maketitle
\tableofcontents

\section{Introduction}

In groundbreaking and classical work, V.I. Smirnov \cite{Smirnov} and
independently A. Beurling \cite{Beurling} characterized the cyclic vectors for
the shift operator, multiplication by $z$, in the Hardy space of the
unit disk $\DD \subset \CC$: a function $f\in H^2(\mathbb{D})$ is
cyclic precisely if it is \emph{outer}, meaning
\[
\log |f(0)| = \int_0^{2\pi} \log |f(e^{i\theta})| \frac{d\theta}{2\pi}
\text{ and } f(0) \ne 0.
\]
We recall that a vector is cyclic for an operator (or operators) if
the invariant subspace generated by the vector is dense.  In this
paper, we generalize in two directions---we move to functions on the
bidisk
\[\mathbb{D}^2=\{(z_1,z_2)\in \mathbb{C}^2\colon |z_1|<1, |z_2|<1\}\]
and we study a whole scale of Hilbert spaces---Bergman, Hardy, and
Dirichlet spaces on the bidisk among them.  At the same time, we also
confine ourselves to characterizing cyclic \emph{polynomials} for
multiplication by the coordinates $z_1,z_2$.

There is an intriguing motivation for generalizing to polydisks.  As
the title of \cite{Nikolskii} suggests, characterizing cyclic vectors
of the Hardy space on the infinite polydisk is ``in the shadow of the
Riemann hypothesis,'' because of the connection between Nyman's
dilation completeness problem (equivalent to the RH) and cyclicity on the
Hardy space of the infinite polydisk.  Characterizing the cyclic
polynomials in the Hardy space of the polydisk $H^2(\DD^n)$ is
straightforward.  Neuwirth, Ginsberg, and Newman \cite{NGN} proved
that a polynomial $f \in \CC[z_1,\dots, z_n]$ is cyclic for
multiplication by the coordinates $z_1,z_2,\dots, z_n$ exactly when
$f$ has no zeros in $\DD^n$; see also \cite{Gel95}.  However, not much
is known with regard to the cyclicity problem for spaces of functions
on the polydisk beyond the Hardy space even when we restrict to the
problem for polynomials.  Here we address this problem on the bidisk
for so-called Dirichlet-type spaces.  We will see that the answer both
depends on the size and nature of the zero sets on the
\emph{distinguished boundary}
\[
\TT^2:= \{(z_1,z_2) \in \CC^2: |z_1|=|z_2|=1\}
\]
of the bidisk.

Let $\alpha\in \RR$ be fixed. A holomorphic function $f\colon \mathbb{D}^2\to \mathbb{C}$
belongs to the {\it Dirichlet-type space} $\mathfrak{D}_{\alpha}$ if the
coefficients $\{a_{k,l}\}_{k,l=0}^{\infty}$ in its Taylor series expansion
\[f(z_1,z_2)=\sum_{k=0}^{\infty}\sum_{l=0}^{\infty}a_{k,l}z_1^kz_2^l\]
satisfy
\begin{equation}
\|f\|^2_{\alpha}=
\sum_{k=0}^{\infty}\sum_{l=0}^{\infty}(k+1)^{\alpha}(l+1)^{\alpha}|a_{k,l}|^2<\infty.
\label{JRnorm}
\end{equation}
It follows directly from this definition that polynomials in two variables form a dense subset of
$\mathfrak{D}_{\alpha}$ for each $\alpha$.

The spaces $\mathfrak{D}_{\alpha}$ generalize
the one-variable Dirichlet-type spaces $D_{\alpha}$, $\alpha\in \RR$,
that consist of functions $f(z)=\sum_{k=0}^{\infty}a_kz^k$ that are
analytic in $\DD$ and satisfy
\[\|f\|^2_{D_{\alpha}}=\sum_{k=0}^{\infty}(k+1)^{\alpha}|a_k|^2 <
\infty.\] These spaces of functions in the unit disk are discussed in
the textbook \cite{EKMRBook}.  Returning to the two-dimensional
setting and putting $\alpha=0$, we recover the {\it Hardy space} of
the bidisk (see \cite{RudBook}). The choice $\alpha=-1$ leads to the
{\it Bergman space} of square area-integrable functions on the bidisk,
while $\alpha=1$ produces the {\it Dirichlet space} of the bidisk.  A
partial list of works addressing $\mathfrak{D}_{\alpha}$ for other
values of $\alpha$ includes \cite{Hed88, Kap94, JR06, BCLSS13II}. As
is pointed out in these references, for $\alpha <2$, these spaces have
the equivalent norm
\begin{equation*}
\begin{aligned}
\|f\|^2_{\alpha,\ast}=&|f(0,0)|^2\\
&+ \int_{\DD}
|\partial_{z_1}f(z_1,0)|^2 w(z_1)^{1-\al} dA(z_1) + \int_{\DD}
|\partial_{z_2}f(0,z_2)|^2 w(z_2)^{1-\al} dA(z_2) \\
&+ \int_{\DD^2} 
|\partial_{z_1 z_2} f(z_1,z_2)|^2 w(z_1)^{1-\al} w(z_2)^{1-\al}
dA(z_1)dA(z_2)
\end{aligned}
\end{equation*}
where $w(z) := (1-|z|^2)$ and $dA$ is area measure on $\DD$.

Evaluation at a point of $\mathbb{D}^2$ is a bounded linear functional
on the Dirichlet-type spaces, and hence the $\mathfrak{D}_{\alpha}$
are {\it reproducing kernel Hilbert spaces} for all $\alpha$. When
$\alpha>1$, the spaces $\mathfrak{D}_{\alpha}$ are in fact {\it
  algebras} contained as subsets of the {\it bidisk algebra}
$A^2(\mathbb{D}^2)$ (see \cite{RudBook}).

\subsection{Cyclic vectors}
Inspecting the definition of the norm in \eqref{JRnorm}, we see that the {\it shift operators}
$S_1$ and $S_2$ defined via
\begin{equation*}
S_1f(z_1,z_2)=z_1f(z_1,z_2) \quad \textrm{and}\quad S_2f(z_1,z_2)=z_2f(z_1,z_2)
\end{equation*}
act boundedly on $\mathfrak{D}_{\alpha}$ as commuting linear operators. 
Coordinate shifts acting on Dirichlet-type spaces represent a concrete and 
natural model of multivariable operator theory on Hilbert space, and the 
{\it lattice of invariant subspaces} of the shift operators acting on 
$\mathfrak{D}_{\alpha}$ exhibits a rich structure that is to a large extent 
still not well understood.

In this paper we are primarily interested in 
invariant subspaces generated by a given function $f\in
\mathfrak{D}_{\alpha}$. 
These subspaces are of the form 
\begin{equation*}
[f]=\overline{\textrm{span}}\{z_1^kz_2^lf\colon k=0,1, \ldots; l=0, 1, \ldots\}
\end{equation*}
where $\overline{\text{span}}$ denotes the closed span of the given
set.  The set $[f]$ is, by definition, the smallest closed subspace
that contains $f$ and is invariant under the action of the shift
operators $S_1$ and $S_2$. A function $f\in \mathfrak{D}_{\alpha}$ for
which $[f]$ coincides with $\mathfrak{D}_{\alpha}$ is called {\it
  cyclic}. In this paper, as in \cite{BCLSS13II}, we address the {\it
  cyclicity problem} for $\mathfrak{D}_{\alpha}$---the problem of
determining the $f\in \mathfrak{D}_{\alpha}$ for which
$[f]=\mathfrak{D}_{\alpha}$.

Cyclicity of the constant functions $f=c$, $c\in \mathbb{C}
\setminus \{0\}$,
follows from the fact that polynomials are dense in $\mathfrak{D}_{\alpha}$. As
point evaluations are bounded linear functionals, cyclicity of $f$ is then
equivalent to the existence of a sequence of polynomials $(p_n)$ such that
\[p_n(z_1,z_2)f(z_1,z_2)-1\longrightarrow 0, \quad (z_1,z_2)\in \mathbb{D}^2,\]
and
\[\|p_nf-1\|_{\alpha}\le C.\]
In view of this, it is clear that functions that have zeros inside the bidisk
cannot be cyclic.  When $\alpha>1$, the spaces $\mathfrak{D}_{\alpha}$ are
 algebras, and functions are cyclic if and only if they are
non-vanishing on $\overline{\mathbb{D}^2}$, the
closure of the bidisk,
but when $\alpha\leq 1$ there exist cyclic functions that vanish on $\partial \mathbb{D}^2$.

In contrast to the Smirnov-Beurling result on $H^2(\DD)$, 
the direct two variable analogue of outer function fails to
characterize cyclicity.

Rudin discovered that any cyclic function in $H^2(\DD^2)$ must
be outer, but there exist non-cyclic outer functions in the
bidisk \cite{RudBook}. 
As mentioned in the introduction, a polynomial is cyclic in
$H^2(\DD^2)$, and hence in $\Dal$ for $\al \leq 0$, if and only if it
has no zeros in the bidisk \cite{NGN}, \cite{Gel95}.  On the other
hand, in \cite{BCLSS13II} the authors exhibited polynomials without
zeros in $\mathbb{D}^2$ that are {\it not} cyclic in
$\mathfrak{D}_{\alpha}$ when $\alpha>1/2$.  Non-cyclicity of certain
polynomials has also been observed by Stefan Richter and Carl Sundberg
\cite{RSslides} in the context of the Drury-Arveson space in the unit
ball of $\mathbb{C}^n$ when $n\geq 4$.

In \cite{BCLSS13II}, certain classes of cyclic functions were
identified using previous work on Dirichlet spaces in the unit disk in
\cite{BCLSS13}.  Indeed, three examples which actually illustrate our
main theorem were discussed.  The polynomials depending on one
variable $z_1-1$ or $z_2-1$ are cyclic exactly when $\al \leq 1$
(exactly as in one variable).  The polynomial $2-z_1-z_2$ which
vanishes at a single point $(1,1)$ of $\TT^2$ is cyclic exactly when
$\al \leq 1$.  Finally, the polynomial $1-z_1z_2$, which does not
depend on one variable and vanishes on a curve in $\TT^2$, is cyclic
exactly when $\al \leq 1/2$.  The problem of characterizing the cyclic
polynomials for each $\mathfrak{D}_{\alpha}$, $\alpha \in (0,1]$, was
stated as an open problem, see \cite[Problem 5.2]{BCLSS13II}.

\subsection{Main result and plan of the paper}
In the present paper, we solve this problem and provide a complete
characterization of the cyclic polynomials in Dirichlet-type spaces of the bidisk. 
Note that since polynomials are multipliers, a polynomial is cyclic if
and only if each of its irreducible factors is cyclic; see Section
\ref{funcspaceprops}.  Therefore, we may state the main theorem as
follows. Below $\ZC(f) = \{z \in \CC^2: f(z) = 0\}$.  

\begin{Mth}[Cyclicity of polynomials in the bidisk]
Let $f(z_1,z_2)$ be an irreducible polynomial with no zeros in the bidisk.
\begin{enumerate}
\item If $\alpha \leq 1/2,$ then $f$ is cyclic in $\mathfrak{D}_{\alpha}$.
\item If $ 1/2 < \alpha \leq 1,$ then $f$ is cyclic in $\mathfrak{D}_{\alpha}$ if and only if $\mathcal{Z}(f)\cap \mathbb{T}^2$ is an empty or finite set, or $f$ 
is a constant multiple of $\zeta - z_1$ or of $\zeta - z_2$ for some $\zeta \in \TT$.
\item If $ \alpha > 1,$ then $f$ is cyclic in $\mathfrak{D}_{\alpha}$ if and only if $\mathcal{Z}(f)\cap \mathbb{T}^2$ is empty.
\end{enumerate}
\end{Mth}

The new and non-trivial content of the theorem is the case $0 <
\alpha \leq 1$ and for polynomials $f$ with
$\mathcal{Z}(f)\cap \partial \DD^2\neq \varnothing$.  Recall that any
polynomial that vanishes {\it inside} the bidisk cannot be cyclic for
any $\alpha$, a polynomial that does not vanish on $\overline{\DD}^2$
is cyclic for all $\alpha$, and polynomials that do not vanish in the
bidisk are known to be cyclic for $\alpha \leq 0$, see \cite{NGN}. In
addition, it is known that a function $f \in \mathfrak{D}_{\alpha}$
for $\alpha > 1$ is cyclic if and only if it does not vanish on the
closure of the bidisk, since in that case $\mathfrak{D}_{\alpha}$ is
an algebra.  

The analogous problem of characterizing cyclicity of polynomials for
Dirichlet-type spaces in one variable was solved by Brown and Shields
\cite{BS84}. In that context, a polynomial with no zeros inside the
disk is always cyclic for $\alpha \leq 1$. The problem without the
assumption of $f$ being a polynomial is already quite difficult in one
dimension, and the so-called \emph{Brown-Shields conjecture}---a
proposed characterization of cyclicity in the Dirichlet space---still
remains open.  Generalizing our theorem to three or more variables
also remains open, as much of our approach in item (1) of the main
theorem relies heavily on an essentially two variable result: the
existence of determinantal representations for polynomials with no
zeros on the bidisk.  See \cite{AM05, GKW, Kne09, VV, HW}.

We now present a plan of the paper. Section \ref{prelim} contains some
preliminary material concerning polynomials in two variables,
including determinantal representation formulas, as well as estimates
on Fourier coefficients of measures supported on curves in the torus,
and a brief discussion of the notion of $\alpha$-capacity and its
relevance to the cyclicity problem.

In Section \ref{finitezeros} we prove that a polynomial that is non-zero in the bidisk and
vanishes at only finitely many points in the distinguished boundary is cyclic for {\it all} $\alpha \leq 1$, see Theorem \ref{thmfinitezeros}. The
proof uses \L ojasiewicz's inequality and results on cyclicity of product functions.

Next, in Section \ref{smallalpha}, we show that any polynomial with
$\mathcal{Z}(f)\cap \DD^2=\varnothing$, whether vanishing at finitely
many points or along a curve, is cyclic whenever $\alpha \leq
1/2$. This is the content of Theorem \ref{smallalphathm}.  The
argument relies on the determinantal formulas mentioned above.

Finally, in Section \ref{Curves}, we use the notion of
$\alpha$-capacity, along with estimates on Fourier coefficients of
measures supported on subvarieties of the torus, to give conditions on
zero sets that imply non-cyclicity in the Dirichlet spaces
$\mathfrak{D}_\alpha$ in the remaining range $1/2< \alpha \leq
1$. Thus, Theorem \ref{bigalphathm} completes the classification
enunciated in the Main Theorem.  In fact, some of the methods in
Section \ref{Curves} apply to more general functions not just
polynomials.

\textit{Acknowledgments.}  Part of this work was done while Sola was
visiting the University of Tennessee, Knoxville in March 2014. Thanks
go to Prof.\! Stefan Richter and the Department of Mathematics for their
hospitality. Kosi\'nski is greatly indebted to Prof.\! Thomas
Ransford for his generosity and for many valuable discussions.

\section{Preliminaries}\label{prelim}
\subsection{Function space properties} \label{funcspaceprops} 

We summarize certain basic facts concerning multipliers and M\"obius
invariance we shall need later on.

First, we note that all polynomials in two variables whose zero sets
do not intersect the bidisk are outer, hence are natural candidates
for being cyclic in $\mathfrak{D}_{\alpha}$ when $\alpha \in [0,1]$.
Furthermore, if $f$ is a polynomial, then $f$ belongs to the {\it
  multiplier space}
\[M(\mathfrak{D}_{\alpha})=\{\phi\colon \mathbb{D}^2\to
\mathbb{C}\colon \phi f\in \mathfrak{D}_{\alpha}, \, \textrm{for
  all}\, f\in \mathfrak{D}_{\alpha}\},\] for all $\alpha$ (see
\cite{JR06} for more on multipliers). In addition, holomorphic
functions that extend to a bigger polydisk also furnish multipliers,
as can be seen by approximating such a function by a finite truncation
of its power series.

We have already seen that constant functions, are cyclic in
$\mathfrak{D}_{\alpha}$. Moreover, since all polynomials are
multipliers, standard techniques (as those used for example in
\cite[Proposition 8]{BS84}) show that it is enough to check cyclicity
for {\it irreducible polynomials} that do not vanish on the bidisk: a
product of multipliers is cyclic if and only if each factor is cyclic.
In particular, multiplicity does not matter: $f^2$ is cyclic whenever
the polynomial $f$ is cyclic.

The automorphism group of the
bidisk consists of the M\"obius transformations
\begin{equation}
m_{a,b}\colon(z_1,z_2)\mapsto \left(\frac{a-z_1}{1-\bar{a}z_1},\frac{b-z_2}{1-\bar{b}z_2}\right),
\label{mobmap}
\end{equation}
where $a,b\in \DD$, as well as rotations of individual variables and
permutation of variables. Using the integral norm for $\Dal$, one can
show that composition with any $m_{a,b} \in
\mathrm{Aut}(\mathbb{D}^2)$ defines a bounded invertible operator, and
hence
\[\|f\circ m_{a,b}\|_{\alpha}\asymp \|f\|_{\alpha}, \quad f\in
\mathfrak{D}_{\alpha}.\] Just as in one variable, \emph{equality} and
not just equivalence of norms for all automorphisms actually
characterizes the Dirichlet space $\mathfrak{D}=\mathfrak{D}_1$ among
Hilbert spaces of analytic functions (see \cite{EKMRBook} and
\cite{Kap94} for details).

\subsection{General remarks on polynomials and their zero varieties} \label{polyremarks}
Here we discuss of polynomials in several variables, their
zero sets
\[\mathcal{Z}(f)=\{(z_1,z_2)\in \mathbb{C}^2\colon f(z_1,z_2)=0\},\]
and factorization properties. 

It is well-known from commutative algebra that $\mathbb{C}[z_1,z_2]$,
the ring of polynomials in two complex variables, is a unique
factorization domain, meaning any polynomial can be written
as a finite product of irreducible polynomials and this
factorization is unique up to units and ordering of irreducible
factors. Whereas the fundamental theorem of algebra asserts that
irreducible polynomials over $\mathbb{C}$ are linear, irreducible
polynomials in two variables come in many different forms.

Let us now closely examine the structure of polynomials that do not
vanish in the bidisk, but have zeros on its boundary.  It suffices to
examine irreducible polynomials which do not depend on one variable
alone, because these polynomials are already addressed in the earlier
work \cite{BCLSS13II}.  So, let $f \in \CC[z_1,z_2]$ be irreducible with no zeros in
$\DD^2$ and assume $f$ depends on both variables.

The set $\mathcal{Z}(f)\cap \mathbb{T}^2$ can be described in terms of
zeros of real polynomials (by considering real and imaginary parts of
$f$), and this implies in particular that $\mathcal{Z}(f)\cap
\mathbb{T}^2$ is a {\it semi-algebraic set}.  By a result by Stanis\l
aw \L ojasiewicz (see \cite[pp.1584-1585]{Loja1993} for a precise
statement, and \cite{Loja1993,KP02} for background material), such
sets split into finitely many connected sets which are again
semi-algebraic.  It follows that the intersection of the zero set of a
polynomial with $\mathbb{T}^2$ consists of finitely many points and a
union of finitely many subvarieties of $\mathbb{T}^2$.  We can say
much more about the zero set of $f$.

First, our irreducible polynomial $f$ cannot vanish on
$\overline{\DD}^2\setminus \TT^2$.  In fact, a zero on $\DD\times \TT$
forces $f$ to vanish on a line $z_2=\zeta$ for some $\zeta \in \TT$
and hence $z_2-\zeta$ would divide $f$, which we have ruled out.
This is because $g_{\zeta} (z_1) := f(z_1,\zeta)$ is non-vanishing in $\DD$ as
a function of $z_1$ for fixed $\zeta \in \DD$, and by Hurwitz's theorem
as $\zeta$ goes to a point on $\TT$, $g_\zeta$ is either non-vanishing
in $\DD$ or identically zero.

Next, if $f$ has bidegree $(n,m)$---degree $n$ in $z_1$ and $m$ in
$z_2$---then we define 
\[
\tilde{f}(z) = z_1^nz_2^m \overline{f(1/\bar{z}_1, 1/\bar{z}_2)}.
\]
There are now two cases depending on whether or not $f$ divides
$\tilde{f}$.  If $f$ does not divide $\tilde{f}$, then $f$ has only
finitely many zeros on $\TT^2$. Indeed, zeros of $f$ on $\TT^2$ will
be common zeros of $f$ with $\tilde{f}$ and two bivariate polynomials
have infinitely many common zeros if and only if they have a common
factor by B\'ezout's theorem.  If $f$ divides $\tilde{f}$, then $f =
\lambda \tilde{f}$ for some $\lambda \in \TT$ since $f$ and
$\tilde{f}$ have the same degree and $|f| = |\tilde{f}|$ on $\TT^2$.
Because of this symmetry, if we set $\EE = \CC\setminus
\overline{\DD}$, then 
\[
\mathcal{Z}(f) \subset (\DD \times \EE) \cup \TT^2 \cup (\EE\times
\DD).
\]

It turns out that our polynomial $f$ possesses the following
determinantal representation
\begin{equation} \label{detpolyformula}
f(z) = c \det\left[ I_{n+m}
  - U \begin{pmatrix}  z_1 I_n & 0 \\ 0 & z_2 I_m\end{pmatrix} \right]
\end{equation}
where $c$ is a constant, $U$ is an $(n+m)\times(n+m)$ unitary matrix,
$I_n,I_m,I_{n+m}$ are identity matrices.  This representation easily
extends to all polynomials non-vanishing on the bidisk all of whose
irreducible factors intersect $\TT^2$ on an infinite set.  This
formula and the techniques used to prove it from \cite{Kne09, Kne10}
will play a fundamental role in our proof in Section
\ref{smallalpha}. The determinantal formula is also useful for
generating concrete examples as in Section \ref{Curves}.

We mention in passing that these polynomials are closely related to a
special class of varieties studied by Agler and M{\McC}Carthy
\cite{AM05} called {\it distinguished varieties}. Such varieties are
given as the zero set $\mathcal{Z}(g)$ of a polynomial $g$ satisfying
\[
\mathcal{Z}(g) \subset \DD^2 \cup \TT^2 \cup \EE^2.
\]
This definition is equivalent to the original definition of Agler and
M{\McC}Carthy and is sometimes easier to work with (see \cite{Kne09}).
The requirement can be phrased as requiring that $\mathcal{Z}(g)$
exits the bidisk through the distinguished boundary.  Since these
varieties must intersect $\DD^2$, they are tangential to the topic of
this paper; however, the formula \eqref{detpolyformula} was discovered
in \cite{Kne09} while giving a new proof of a similar formula for
distinguished varieties from \cite{AM05}.

\subsection{Fourier analysis and measures supported on curves}\label{ss-oscillatory}
We will use a generalization of van der Corput's lemma from the theory
of oscillatory integrals (see Stein \cite[Section
VIII.3.2]{SteinHarmAnaBook}) that provides estimates on Fourier
coefficients of measures supported on submanifolds. Such estimates
will be useful in Section \ref{Curves} to establish existence of
measures of finite energy.

Let $S$ be a smooth curve
in $\mathbb{T}^2$, which we identify with $[0,2\pi)\times [0,2\pi)$ in
this section. Consider a smooth parametrization
$\phi: I\to \mathbb{T}^2$ where $I$ is an interval
in $\mathbb{R}$ (and $d\phi/dt\neq 0$ on $I$). We define the \emph{type} of a point
$\xi=\phi(x)\in \phi(I)$ as the smallest $\tau$ such that, for all unit
vectors $\eta\in \mathbb{R}^2$ there exists an integer $k \leq \tau$ such that we have
\begin{align}\label{e-type}
\left[\frac{d^k\phi}{dt^k}\cdot\eta\right]_{t=x} \neq 0.
\end{align}
A curve $S$ is said to have type $\tau$ if the maximum of the types of
its points is $\tau$. In particular, a curve is of type $2$ if it has
everywhere non-vanishing curvature.

In Section \ref{Curves}, we discuss an example in detail, however it
is useful to discuss the type of a curve which is simply a piece of a
graph $\phi(t) = (t, \psi(t))$ with $\psi'(0)\ne 0$.  If $\eta \in
\RR^2$ is a unit vector with $\phi'(t) \cdot \eta =
\eta_1+\psi'(0)\eta_2 = 0$, then $\eta_2 \ne 0$. If we also had $0 =
\phi''(0) \cdot \eta = \psi''(0) \eta_2$, then necessarily $\psi''(0)
= 0$.  Thus, the curve has type 2 at $t=0$ if $\psi''(0) \ne 0$ and it
will not have type 2 if $\psi''(0) = 0$.  It will be useful later to
notice that if we do not have type 2 at a point we can apply a
M\"obius transformation to get a point of type 2.  What we have in
mind is if $\phi$ parametrizes a piece of the zero set of $f$ on
$\TT^2$, then $(\gamma(t), \psi(t))$ parametrizes a piece of the zero
set of $f \circ m_{a,0}$ where $\gamma(t) = \arg m_a(e^{it})$.  But,
this new curve has type 2 at $t=0$.  One can compute that $\gamma'(0)
>0$ and $\gamma''(0) \ne 0$ as long as $\text{Im}(a) \ne 0$.  So, if
$\psi''(0) = 0$, the two equations $\eta_1 \gamma'(0) + \eta_2
\psi'(0) = 0$ and $\eta_1 \gamma''(0) = 0$ cannot hold simultaneously.
The type of a curve will give us control on the growth of Fourier
coefficients of a measure supported on the curve.

The Fourier coefficients of a finite Borel measure on $\mathbb{T}^2$ are given by
\begin{align}\label{Fourier}
\hat{\mu}(k,l)=\int_{\mathbb{T}^2}e^{-i(k\theta_1+l\theta_2)}d\mu(\theta_1,\theta_2),
\quad k,l \in \mathbb{Z}.
\end{align}
The measures we work with will be of the form
\begin{equation}
d\mu(x)=\psi(x)d\sigma(x), \quad x\in S\subset \mathbb{T}^2,
\label{smoothmeasure}
\end{equation}
where $S$ is a curve in $\mathbb{T}^2$ of finite type, $\psi\in
C^{\infty}_0(S)$ is non-negative, and $\sigma$ is the measure on $S$
induced by pulling back to Lebesgue measure on the line using the
parametrization of $S$. For such measures, we have the following
estimate (cf. Stein \cite{SteinHarmAnaBook}, Theorem VIII.2).
\begin{thm}[Decay of Fourier coefficients of measures on varieties]\label{t-oscillatory}
If $S$ is of finite type $\tau\in \mathbb{N}$, and $\mu$ is of the form \eqref{smoothmeasure}, then there exists a constant $C>0$ such that
\begin{align*}
|\hat{\mu}(k,l)|\le C (k^2+l^2)^{-1/(2\tau)}, \quad k,l\in \mathbb{Z}.
\end{align*}
The constant $C$ may depend on $S$ and $\psi$, but not on $k$ and $l$.
\end{thm}

\subsection{Riesz $\alpha$-capacity}\label{alphacapa}
In Section 5, we will use an analogue of logarithmic capacity in one variable to analyze the cyclicity properties of polynomials whose zero sets meet the torus along curves.  We define the following product capacity on $\mathbb{T}^2$, which was
considered in \cite{BCLSS13II}.

\begin{defn}
Let $E\subset \mathbb{T}^2$ be a Borel set. We say that a probability
measure $\mu$ supported in $E$ has {\it finite (Riesz) $\alpha$-energy} ($0<\alpha<1$) if
\begin{equation*}
I_{\alpha}[\mu]=\int_{\mathbb{T}^2}\int_{\mathbb{T}^2}\frac{1}{|e^{i\theta_1}-e^{i\eta_1}|^{1-\alpha}}\frac{1}{|e^{i\theta_2}-e^{i\eta_2}|^{1-\alpha}}d\mu(\theta_1,\theta_2)d\mu(\eta_1,\eta_2)<\infty.
\label{logenergy}
\end{equation*}
We define the \emph{(Riesz) $\alpha$-capacity} by taking the infimum
\begin{align*}
\alphacap(E):= 1/\inf\{I_{\alpha}[\mu]: \mu\in \mathcal{P}(E)\},
\end{align*}
where $\mathcal{P}(E)$ denotes the set of all probability measures with support contained in $E$.
If $E$ supports no such measure, then $\alphacap(E)=0$.

When $\alpha=1$, we use the kernel $\log(e/|e^{i\theta_1}-e^{i\eta_1}|)\log(e/|e^{i\theta_2}-e^{i\eta_2}|)$
in the definitions of energy and capacity.
\end{defn}
In \cite[Section 4]{BCLSS13II}, it was shown that functions
$f\in\mathfrak{D}_\alpha$ with
$\mathrm{cap}_{\alpha_0}(\mathcal{Z}(f)\cap \TT^2)>0$ for some
$0<\alpha_0<1$ are not cyclic in $\mathfrak{D}_\alpha$ for all
$\alpha\ge \alpha_0$.  Here, $\mathcal{Z}(f)\cap\TT^2$ refers to the
zero set of $f$ on $\TT^2$.  (As explained in \cite{BCLSS13II}, the
non-tangential boundary values of $f$ on $\TT^2$ exist
quasi-everywhere with respect to $\al$-capacity so that we can make
sense of the capacity of $\mathcal{Z}(f)\cap\TT^2$.)  In brief, the
proof involves identifying the dual space of $\mathfrak{D}_{\alpha}$
with $\mathfrak{D}_{-\alpha}$ and considering the Cauchy integral of a
measure $\mu$ supported on $\mathcal{Z}(f)\cap \TT^2$ having finite
$\alpha$-energy.  The finiteness of the energy then guarantees the
membership of the Cauchy transform in $\mathfrak{D}_{-\alpha}$, and
the fact that its generating measure lives on the zero set of $f$
implies that the dual pairing annihilates the elements of $[f]$.

For two Borel sets $E, F\subset \mathbb{T}^2$ with $E\subset F$ we have $\mathcal{P}(E)\subset \mathcal{P}(F)$, and so $\alphacap(E)\le\alphacap(F)$. This means that if we want to show that a given zero set has positive capacity, it is enough to prove that some subset of the given zero set has this property. It is often useful to express the $\alpha$-energy of $\mu$ in terms of its Fourier coefficients \eqref{Fourier}:
\begin{equation}
I_{\alpha}[\mu]\asymp 1+\sum_{k=1}^{\infty}\frac{|\hat{\mu}(k,0)|^2}{k^{\alpha}}
+\sum_{l=1}^{\infty}\frac{|\hat{\mu}(0,l)|^2}{l^{\alpha}}+\frac{1}{2}\sum_{k \in \mathbb{Z}\setminus \{0\}}
\sum_{l=1}^{\infty}\frac{|\hat{\mu}(k,l)|^2}{|k|^{\alpha} l^{\alpha}}.
\label{logenergycoeffs}
\end{equation}
\section{Polynomials with finitely many zeros}\label{finitezeros}
We now turn to the main result. We begin by tackling the case when $f$ is a polynomial whose zero set satisfies $0 < \#[\mathcal{Z}(f)\cap \mathbb{T}^2]<\infty$.

Building on the work of H\aa kan Hedenmalm on closed ideals in $\mathfrak{D}_2$ in \cite{Hed88}, the authors in \cite{BCLSS13II} proved that polynomials (and in fact, certain functions in $\mathfrak{D}_2$) with no zeros in the bidisk that vanish at a single point of  the torus are cyclic in $\mathfrak{D}_{\alpha}$ for all $\alpha\leq 1$.  Hedenmalm's proof does not seem to extend directly to the case of finitely many zeros.  Instead, we resolve the question
of cyclicity for polynomials with finitely many zeros on the distinguished boundary using a different argument.

\begin{thm} \label{thmfinitezeros}
Let $f \in \CC[z_1,z_2]$ have no zeros in $\DD^2$ and finitely many
zeros on $\TT^2$.  Then, $f$ is cyclic in $\mathfrak{D}_{\alpha}$ for
$\alpha \leq 1$.
\end{thm}

We shall use \L ojasiewicz's inequality as recorded in \cite{KP02}.

\begin{thm}[\L ojasiewicz's inequality]
Let $f$ be a nonzero real analytic function on an open set $U\subset
\RR^n$.  Assume the zero set $\mathcal{Z}(f)$ of $f$ in $U$ is nonempty. Let $E$ be a compact subset of $U$.  Then there are constants $C >0$
and $q \in \NN$, depending on $E$, such that
\[
|f(x)| \geq C \cdot \dist (x,\mathcal{Z}(f))^q
\]
for every $x \in E$.
\end{thm}

To prove our theorem we will essentially compare our polynomial $f$ to
polynomials which we know to be cyclic.

\begin{lem}\label{product}
  Assume $f \in \CC[z_1,z_2]$ has no zeros in $\DD^2$ and finitely many zeros on $\TT^2$.
  Then, for any positive integer $k$ there exist $g \in \CC[z_1]$ and
  $h \in \CC[z_2]$ with zeros only on $\TT$ such that $Q$ defined by
\[
Q(z_1,z_2) = \frac{g(z_1)h(z_2)}{f(z_1,z_2)}
\]
is $k$-times continuously differentiable on $\TT^2$.
\end{lem}

\begin{proof}
  Let $r(x) = r(x_1,x_2) = |f(e^{ix_1}, e^{ix_2})|^2$, a function which is
  evidently real analytic throughout $\RR^2$, and let $\ZC(r)$ be the set
  of zeros of $r$.  Set $E = [0,2\pi]^2$.  By \L ojasiewicz's
  inequality there is a constant $C>0$ and positive integer $q$ so
  that
\[
r(x) \geq C \dist(x, \ZC(r))^q
\]
for $x \in E$.

By the assumption on $f$, $\ZC(r) \cap E$ is finite and thus there is
a constant $c>0$ so that for $x\in E$
\[
\dist(x,\ZC(r))^2  \geq c \prod_{y \in \ZC(r) \cap E} |x-y|^2.
\]
But $|x-y|^2 = |x_1-y_1|^2 + |x_2-y_2|^2\geq |e^{ix_1} -
e^{iy_1}|^2+|e^{ix_2} - e^{iy_2}|^2 \geq 2|(e^{ix_1} -
e^{iy_1})(e^{ix_2} - e^{iy_2})|$.  
Replacing $(e^{ix_1},e^{ix_2})$ by $(z_1,z_2)$ and $(e^{iy_1},e^{iy_2})$ by $(\zeta_1,\zeta_2),$ we get that
\[
\frac{\prod_{\zeta \in \ZC(f)\cap \TT^2}
  |(z_1-\zeta_1)(z_2-\zeta_2)|^{q/2}}{|f(z_1,z_2)|^2}
\]
is bounded on $\TT^2\setminus \ZC(f)$.  If we increase the power in
the numerator (say to $4q$) we get a function which is continuous on
$\TT^2$ and set equal to zero at the zeros of $f$.  Thus, for
\[
Q_0(z) = \prod_{\zeta \in \ZC(f)\cap \TT^2}
(z_1-\zeta_1)^q(z_2-\zeta_2)^q
\]
we have that $Q_0/f$ is bounded and continuous on $\TT^2$ (where again
we set the function equal to zero at zeros of $f$).  Therefore, for a
large enough power $N$, $Q_0^N/f$ is $k$-times continuously
differentiable on $\TT^2$. By construction, $Q_0^N(z_1,z_2) =
g(z_1)h(z_2)$ for some one variable polynomials $g,h$ which only
vanish on the circle.
\end{proof}

\begin{proof}[Proof of Theorem \ref{thmfinitezeros}]
  By Lemma \ref{product}, we have one variable polynomials $g,h$ which
  vanish only on the circle such that $Q(z_1,z_2) =
  g(z_1)h(z_2)/f(z_1,z_2)$ is $2$-times continuously differentiable on
  $\TT^2$.  Then, the Fourier coefficients of $Q$ satisfy
\[
\sum_{k,l} |\hat{Q}(k,l)|^2 (k+1)^2 (l+1)^2 <\infty
\]
which puts $Q \in \mathfrak{D}_{\alpha}$ for $\alpha \leq 2$.  Hence,
$g(z_1)h(z_2) \in f \mathfrak{D}_{\alpha}$ for $\alpha \leq 2$. Since $f$ is a multiplier, 
 $f \mathfrak{D}_{\alpha} = [f]$. 
But $g(z_1)h(z_2)$ is cyclic for $\alpha$ with $\alpha \leq 1$ (by \cite{BCLSS13II}), and therefore $f$ is cyclic for
$\alpha \leq 1$. \end{proof}

\section{Cyclicity for small parameter values}\label{smallalpha}
In this section, we prove that any polynomial that does not vanish in
the bidisk is cyclic for $\alpha\leq 1/2$, regardless of the size of
its zero set.  As was mentioned previously, the case $\alpha \leq 0$
follows from \cite{NGN}.  

\begin{thm}\label{smallalphathm}
Let $0< \alpha \leq 1/2$.
Then any polynomial that does not vanish in $\DD^2$ is cyclic in $\mathfrak{D}_{\alpha}$.
\end{thm}

The proof presented here relies heavily on the machinery developed in
\cite{Kne09, Kne10}.  As discussed in Section \ref{polyremarks}, we
can assume $f\in \CC[z_1,z_2]$ is
\begin{enumerate}
\item non-vanishing in $\DD^2$,
\item irreducible, 
\item a function of both variables, 
\item of bidegree $(n,m)$, and 
\item $f = \lambda \tilde{f}$ for some $\lambda \in \TT$.
\end{enumerate}
Here $\tilde{f}(z,w) = z_1^n z_2^m \overline{f(1/\bar{z}_1,
  1/\bar{z}_2)}$. The determinantal formula \eqref{detpolyformula} is
a consequence of the following key result which we explain at the end
of this section.

\begin{prop} \label{backgroundprop} Assume $f\in \CC[z_1,z_2]$
  satisfies (1)-(5) above.  Then, there exists an $(n+m)\times(n+m)$
  unitary $U$ and a column vector polynomial $\vec{B} \in
  \CC^{n+m}[z_1,z_2]$ such that
\[
\left(I-U \begin{pmatrix} z_1 I_n & 0 \\ 0 & z_2
  I_m \end{pmatrix}\right) \vec{B}(z) \in f(z)  \CC^{n+m}[z_1,z_2].
\]
Furthermore, there exists a row vector polynomial $\vec{a} \in
\CC^{n+m}[z_2]$ such that $p(z_2):= \vec{a}(z_2) \vec{B}(z_1,z_2)$ is a
one variable polynomial with no zeros on $\DD$.  
\end{prop}

Armed with this proposition, the proof of cyclicity is
straightforward.

\begin{proof}[Proof of Theorem~\ref{smallalphathm}]
Assume $f,\vec{B}, U$ are as in
  Proposition \ref{backgroundprop} and let $\alpha \leq 1/2$.  Suppose
  $g \in \Dal$ is orthogonal to $[f]$ in $\Dal$.  We shall show first that for every row
  vector polynomial $\vec{v} \in \CC^{n+m}[z_1,z_2]$
\[
\ip{\vec{v} \vec{B} }{g}_{\al} = 0.
\]

For simplicity, set 
\[A(z) = U \begin{pmatrix} z_1 I_n & 0 \\ 0 & z_2
  I_m \end{pmatrix} \]
The important thing to notice about $A$ is that it is homogeneous of degree $1$.

Since $(I-A(z)) \vec{B}(z) \in f \CC^{n+m}[z_1,z_2]$ it follows that
$(I-A(z)^k) \vec{B}(z) \in f \CC^{n+m}[z_1,z_2]$ for any $k \geq 1$.
So, if $g \perp [f]$ in $\Dal$, then
\begin{equation} \label{recur}
\ip{\vec{v} \vec{B}}{g}_\al = \ip{\vec{v} A^k \vec{B}}{g}_{\al}
\end{equation}
for any $k\geq 0$.  

Let $d = \deg \vec{v} + \deg \vec{B} + 1$ where the degree of a vector
polynomial is the maximum of the degrees of each entry. We
emphasize we are using total degree and not bidegree.  Then, the
monomials appearing in $\vec{v} A^k \vec{B}$ have total degree between $k$
and $k+d-1$.  So, $\{\vec{v} A^{kd} \vec{B} \}$ forms a sequence of pairwise
orthogonal elements of $\Dal$ since the degrees of the monomials in $\vec{v}
A^{kd} \vec{B}$ lie in the interval $[kd, (k+1)d)$.

By Bessel's inequality
\[
\|g\|^2_{\al} \geq \sum_{k\geq 0} \frac{|\ip{\vec{v} A^{kd}
    \vec{B}}{g}_{\al}|^2}{\|\vec{v} A^{kd} \vec{B}\|^2_{\al}} 
= \sum_{k\geq 0} \frac{|\ip{\vec{v}\vec{B}}{g}_{\al}|^2}{\|\vec{v} A^{kd} \vec{B}\|^2_{\al}}. 
\]
Notice if one of the denominators happened to be zero, then our goal
$\ip{\vec{v} \vec{B}}{g}_{\al} = 0$ holds automatically by \eqref{recur}.
So, we can assume $\vec{v} A^k \vec{B} \ne 0$ for all $k \geq 0$ and the
goal now is to show the denominator grows too slowly for the above sum
to be finite without having $\ip{\vec{v} \vec{B}}{g}_{\al} = 0$.

Since $\vec{v} A^{kd} \vec{B}$ has degree at most $(k+1)d$, 
\[
\begin{aligned}
\|\vec{v} A^{kd} \vec{B}\|^2_{\al} &\leq \max\{ (i+1)^{\al}(j+1)^{\al}: i+j
\leq (k+1)d \} \|\vec{v} A^{kd}\vec{B} \|^2_{H^2} \\
& \leq ((k+1)d/2+1)^{2\al} \|\vec{v}
A^{kd} \vec{B} \|^2_{H^2}. 
\end{aligned}
\]
But $\|\vec{v} A^{kd} \vec{B}\|_{H^2} \leq \|\vec{v}\|_{H^{\infty}}
\|\vec{B}\|_{H^2}$; since $A$ is unitary valued on $\TT^2$ we have
$\|A^{kd}\|_{H^{\infty}} = 1$ for all $k$.  We are implicitly using
vector-valued $H^2$ and matrix valued $H^\infty$ as well as the fact
that $H^\infty$ is the multiplier algebra of $H^2$.  Therefore,
\[
\|\vec{v} A^{kd} \vec{B}\|^2_{\al} \leq C (k+1)^{2\al}
\]
for some constant $C$ depending on $d,\vec{v},\vec{B}$.  
Thus,
\[
\sum_{k\geq 0} \frac{|\ip{\vec{v}\vec{B}}{g}_{\al}|^2}{\|\vec{v} A^{kd}
  \vec{B}\|^2_{\al}} \geq \sum_{k\geq 0}
\frac{|\ip{\vec{v}\vec{B}}{g}_{\al}|^2}{ C(k+1)^{2\al}}.
\]
The only way this can be finite when $\al \leq 1/2$ is if
$\ip{\vec{v}\vec{B}}{g}_{\al} = 0$.

Now we can finish the proof.
Setting $\vec{v}(z) = z_1^j z_2^k \vec{a}(z_2)$ where $\vec{a}$
is as in Proposition \ref{backgroundprop}, we see that for all
$j,k\geq 0$
\[
0 = \ip{z_1^j z_2^k \vec{a} \vec{B}}{g}_{\al} = \ip{z_1^j z_2^k
  p}{g}_{\al}
\]
where $p=\vec{a}\vec{B} \in \CC[z_2]$ from Proposition \ref{backgroundprop} has no
zeros in $\DD$.  Thus, $g$ is orthogonal to $[p]$, but since since $p$
is cyclic for $\al \leq 1$ we have $[p]=\Dal$.  Hence, $g =0$ and we
conclude $f$ is cyclic.
\end{proof}

We now turn to explaining Proposition \ref{backgroundprop}.  This
result can be essentially proven by quoting results from
\cite{Kne09} but some further explanation is necessary.  

First, we can assume $f = \tilde{f}$ by replacing $f$ with a
unimodular multiple.  The condition $f=\tilde{f}$ is what is termed
``$\TT^2$-symmetric'' in \cite{Kne09}.  Many of the results we use
from \cite{Kne09} require this of $f$.  By Theorem 2.9 of
\cite{Kne09}, if we set
\[
h =z_1 \frac{\partial f}{\partial z_1}+
z_2 \frac{\partial f}{\partial z_2}
\]
then $\tilde{h}$ has no zeros in $\DD^2$.  Furthermore, by computation
or by Lemma 2.7 of \cite{Kne09}, $h + \tilde{h} = (n+m)f$.  Since $f$
is irreducible, $h$ and $\tilde{h}$ have no factors in common unless
they are multiples of one another by looking at degrees.  This
possibility is disallowed by $h(0,0)=0$.  Thus, $\tilde{h}$ has no
zeros in $\DD^2$ and no factors in common with $h$.

By the main result of \cite{Kne10}, there exist vector polynomials
$\vec{P}\in \CC^{n}[z_1,z_2]$ and $\vec{Q} \in \CC^{m}[z_1,z_2]$ such
that
\begin{equation} \label{Agdecomp}
\tilde{h}(z) \overline{\tilde{h}(w)} - h(z) \overline{h(w)} =
(1-z_1\bar{w}_1) \vec{P}(w)^* \vec{P}(z) + (1-z_2\bar{w}_2)
\vec{Q}(w)^* \vec{Q}(z).
\end{equation}
Furthermore, we may choose $\vec{P}$ to be of degree $n-1$ in $z_1$
and when we write
\begin{equation} \label{Pfactor}
\vec{P}(z_1,z_2) = P(z_2) \begin{pmatrix} 1 \\ z_1 \\ \vdots \\
  z_1^{n-1} \end{pmatrix} 
\end{equation}
the $n\times n$ matrix $P(z_2)$ is invertible for $z_2 \in \DD$.
Lemma 2.8 of \cite{Kne09}, or a computation, shows the left hand side
of \eqref{Agdecomp} is equal to
\[
(n+m)^2 f(z) \overline{f(w)} - (n+m)( h(z) \overline{f(w)} + f(z)
\overline{h(w)})
\]
---the main point here is that this vanishes on $\ZC(f)$.  Thus,
for $z,w \in \ZC(f)$ equation \eqref{Agdecomp} rearranges into
\[
 \vec{P}(w)^* \vec{P}(z) + \vec{Q}(w)^* \vec{Q}(z) = z_1\bar{w}_1 \vec{P}(w)^* \vec{P}(z) + z_2\bar{w}_2
\vec{Q}(w)^* \vec{Q}(z)
\]
and so the map 
\[
\begin{pmatrix} z_1\vec{P}(z) \\ z_2 \vec{Q}(z) \end{pmatrix}
\mapsto \begin{pmatrix} \vec{P}(z) \\ \vec{Q}(z) \end{pmatrix}
\]
defined for $z \in \ZC(f)$ extends linearly to an isometry which in
turn can be extended to a mapping given by an $(n+m)\times (n+m)$
unitary matrix $U$:
\[
U \begin{pmatrix} z_1\vec{P}(z) \\ z_2 \vec{Q}(z) \end{pmatrix}
= \begin{pmatrix} \vec{P}(z) \\ \vec{Q}(z) \end{pmatrix}
\]
for $z \in \ZC(f)$. Thus, if we set $\vec{B} = \begin{pmatrix} \vec{P}
  \\ \vec{Q} \end{pmatrix}$, then
\begin{equation} \label{Ueq}
\left(I-U \begin{pmatrix} z_1 I_n & 0 \\ 0 & z_2 I_m \end{pmatrix}
\right) \vec{B}(z)
\end{equation}
vanishes on $\ZC(f)$ and since $f$ is irreducible its entries must
belong to the polynomial ideal generated by $f$.  Therefore,
\eqref{Ueq} belongs to $f \CC^{n+m}[z_1,z_2]$.  

To finish we must find a row vector polynomial $\vec{a} \in
\CC^{n+m}[z_2]$ such that $\vec{a}(z_2) \vec{B}(z_1,z_2)$ is a one variable
polynomial in $z_2$ with no zeros in $\DD$.  Set $\vec{a}(z_2) = (e_1
\text{adj}(P(z_2)), 0_m)$ where $e_1 = (1,0,\dots,0) \in \CC^n$ and
$0_m \in \CC^m$ is the zero (row) vector.  Here $\text{adj}$ denotes
the adjugate matrix or the transpose of the matrix of cofactors.
Then, by \eqref{Pfactor}
\[
\vec{a}(z_2)\vec{B}(z) = e_1 \text{adj}(P(z_2)) \vec{P}(z) = e_1
\text{adj}(P(z_2)) P(z_2) \begin{pmatrix} 1 \\ z_1 \\ \vdots \\
  z_1^{n-1} \end{pmatrix} = \det(P(z_2))
\]
which has no zeros in $\DD$.

\section{Non-cyclicity and curves in the distinguished boundary}\label{Curves}
We now turn to the remaining case where $\alpha>1/2$ and $f$ is a
polynomial whose zero set meets the distinguished boundary along
curves. Here, we find that local curvature properties determine the
parameters $\alpha$ for which the function is not cyclic in
$\mathfrak{D}_\alpha$. Our methods are analytic in nature, and do not
rely on the algebraic properties of polynomials to the same extent as
in previous sections.  In particular, the next theorem holds for a
general $f \in \Dal$ which is not necessarily a polynomial.  Recall
that for $\al \geq 0$, any $f \in \Dal$ has non-tangential boundary
values on $\TT^2$ quasi-everywhere with respect to $\al$-capacity and
thus the zero set $\ZC(f) \cap \TT^2$ of $f$ restricted to $\TT^2$ is
well-defined up to sets of $\al$-capacity zero; see \cite{Kap94}.

\begin{thm}\label{t-result}
Assume that $f\in\mathfrak{D}_\alpha$ is such that the intersection $\mathcal{Z}(f) \cap\mathbb{T}^2$ contains a locally smooth curve $S$ of type $\tau$. Then $f$ is not cyclic in $\mathfrak{D}_\alpha$ for any $\alpha>1-1/\tau$.
\end{thm}
\begin{proof}
Since $S\subset \mathcal{Z}(f)\cap\mathbb{T}^2$, we have $\alphacap(S)\le \alphacap(\mathcal{Z}(f)\cap\mathbb{T}^2)$.
Our goal is to show that $\alphacap(S)>0$ for $\alpha>1-1/\tau$ by proving the existence of a measure $\mu$ of the form \eqref{smoothmeasure} that has finite energy. This is accomplished by estimating the Fourier coefficients $\hat\mu(k,l)$ of such a measure, and plugging these estimates into the Fourier formula for the energy given in \eqref{logenergycoeffs}:
\begin{equation*}
I_{\alpha}[\mu]\asymp 1+\sum_{k=1}^{\infty}\frac{|\hat{\mu}(k,0)|^2}{k^{\alpha}}
+\sum_{l=1}^{\infty}\frac{|\hat{\mu}(0,l)|^2}{l^{\alpha}}+\frac{1}{2}\sum_{k \in \mathbb{Z}\setminus \{0\}}
\sum_{l=1}^{\infty}\frac{|\hat{\mu}(k,l)|^2}{|k|^{\alpha} l^{\alpha}}.
\end{equation*}

Applying the estimate in Theorem \ref{t-oscillatory} to the sum of diagonal terms ($k=l\ge 1$) we obtain
\begin{align*}
\sum_{k=1}^\infty \frac{|\hat\mu(k,k)|^2}{(k+1)^{2\alpha}}
&\le C
\sum_{k=1}^\infty \frac{1}{(2k^2)^{1/\tau}(k+1)^{2\alpha}}
\le C
\sum_{k=1}^\infty \frac{1}{k^{2\alpha+2/\tau}}.
\end{align*}
The latter series converges when $\alpha> 1/2 - 1/\tau$. Moreover, we have
\[\sum_{k=1}^{\infty}\frac{|\hat{\mu}(k,0)|^2}{(k+1)^{\alpha}}\leq C\sum_{k=1}^{\infty}\frac{1}{k^{\alpha+1/\tau}},\]
and similarly for the series involving $\hat{\mu}(0,l)$, and having $\alpha>1-1/\tau$ forces convergence.

We perform a similar computation for the remaining terms. We now have to sum over both $k$ and $l$; however, using the symmetry in Theorem \ref{t-oscillatory}, we obtain
\begin{align*}
&\sum_{k=2}^\infty\sum_{l=1}^{k-1} \frac{|\hat\mu(k,l)|^2}{(k+1)^{\alpha}(l+1)^{\alpha}}
+\sum_{l=2}^\infty\sum_{k=1}^{l-1} \frac{|\hat\mu(k,l)|^2}{(k+1)^{\alpha}(l+1)^{\alpha}}\\
\le
&\,\,C
\sum_{k=2}^\infty\sum_{l=1}^{k-1} \frac{1}{(k+l)^{2/\tau}(k+1)^{\alpha}(l+1)^{\alpha}}\\
\le
&\,\,C
\sum_{k=2}^\infty \frac{1}{k^{2/\tau}(k+1)^{\alpha}}\sum_{l=1}^{k-1} \frac{1}{(1+l/k)^{2/\tau}(l+1)^{\alpha}}\\
\le
&\,\,C
\sum_{k=2}^\infty \frac{1}{k^{\alpha+2/\tau}}\sum_{l=1}^{k-1} \frac{1}{(l+1)^{\alpha}}\\
\le
&\,\,C
\sum_{k=1}^\infty \frac{1}{k^{\alpha+2/\tau}}\cdot k^{1-\alpha}\qquad\qquad\qquad\quad(\text{since }1-\alpha>0)\\
\le
&\,\,C
\sum_{k=1}^\infty \frac{1}{k^{2\alpha-1+2/\tau}}\,.
\end{align*}
The latter sum is finite if and only if $2\alpha-1+2/\tau>1$, which happens precisely if $\alpha>1-1/\tau$.

In conclusion, the total energy of $\mu$ is finite if
$
\alpha>
1-1/\tau$, and non-cyclicity of $f$ follows.
\end{proof}
We illustrate the content of the theorem, and shed new light on the results obtained in \cite{BCLSS13II}, by looking at a family of examples.
\begin{ex}
For $a, b \in \CC$ with $|a|^2+|b|^2=1$, we consider the unitary matrix
\[ U=\left(\begin{array}{cc}a & -\overline{b}\\ b& \overline{a}\end{array}\right).\]
Taking $n=2$ and $j=k=1$ in the formula \eqref{detpolyformula}, and plugging in this choice of $U$, we obtain the family of polynomials
\begin{equation}
f_{a}(z_1,z_2)=1-az_1-\overline{a}z_2+z_1z_2.
\label{apolyfamily}
\end{equation}

Setting $a=0$, we obtain $f_0(z_1,z_2)=1+z_1z_2$, an irreducible
polynomial whose zero set and cyclic properties were studied in
\cite[Section 4]{BCLSS13II}. In particular, it was established that
$f_0$ is cyclic in $\mathfrak{D}_{\alpha}$ precisely when $\alpha\leq
1/2$.  The choice $a=1$ leads to the polynomial
$f_1(z_1,z_2)=1-z_1-z_2+z_1z_2=(1-z_1)(1-z_2)$, which factors into
one-variable polynomials; such polynomials were shown to be cyclic for
all $\alpha\leq 1$ in \cite[Section 2]{BCLSS13II}.

Solving the equation $f_a(z_1,z_2)=0$
for $z_2$, we obtain
\[z_2=\frac{az_1-1}{z_1-\bar{a}}.\] Direct computation (or the
observation that the formula for $1/z_2$ describes a M\"obius
transformation) shows that $|z_2|=1$ when $|z_1|=1$.  This leads to a
parametric representation of $\mathcal{Z}(f_a)\cap \mathbb{T}^2$:
\begin{equation}
t\mapsto (e^{it}, e^{i m(t)}),\quad
\textrm{where} \quad
m(t)=\arg\left(\frac{ae^{it}-1}{e^{it}-\overline{a}}\right),
\label{moebzeros}
\end{equation}
or, if the torus is identified with $[0,2\pi)\times [0, 2\pi)$,
$t\mapsto (t, m(t))$.  Restricting to $a \in (0,1)$ for simplicity, we
see that $\mathcal{Z}(f_a)$ ``interpolates'' between
\[
\mathcal{Z}(f_0)=\{(e^{it}, -e^{-it})\} \quad \textrm{and}\quad
\mathcal{Z}(f_1)=(\mathbb{T}\times \{1\})\cup (\{1\}\times \mathbb{T}). \]

In \cite{BCLSS13II}, it was shown that $\mathcal{Z}(f_0)\cap \mathbb{T}^2$ has positive $1/2$-capacity, while $\mathcal{Z}(f_1)\cap\mathbb{T}^2$ has $\alpha$-capacity $0$ for all $\alpha\leq 1$. We will now show that the $\alpha$-capacity of $\mathcal{Z}(f_a)\cap \mathbb{T}^2$ for intermediate $a\in (0,1)$ is positive for all  
$1/2<\alpha \le 1$, so that $f_a$ is not cyclic in $\mathfrak{D}_{\alpha}$ for this range of parameter values.

We show that $\mathcal{Z}(f_a)$ is of type $\tau=2$ at $ t=\pi/2$ when $a\in (0,1)$.
Computations yield
\begin{align*}
m(  t) &= \pi+\arctan\left(\frac{(1-a^2)\sin  t}{2a-(1+a^2)\cos  t}\right), \\
m'(  t) &= \frac{1-a^2}{2a \cos  t -1-a^2},\\
\textrm{and} \,\,\, m''(  t) &= \frac{2a(1-a^2)\sin  t}{(2a \cos  t -1-a^2)^2}\,.
\end{align*}
If  $\eta = (\eta_1,\eta_2)^\top\in \mathbb{R}^2$ is a unit vector such that
\[\left[\frac{d}{dt}(t,m(t))\right]_{t=\pi/2} \cdot \eta= \eta_1  -\frac{1-a^2}{1+a^2}\,\eta_2 = 0, \]
then $\eta$ is determined up to sign. If $a\neq 0,1$, then we cannot simultaneously have
\[\left[\frac{d^2}{dt^2}(t,m(t))\right]_{t=\pi/2} \cdot \eta= -\frac{2a(1-a^2)}{(1+a^2)^2} \,\eta_2 =0,\]
and hence $\mathcal{Z}(f_a)$ is of type $2$ at this point.

By smoothness of $\mathcal{Z}(f_a)$, for fixed $a$, there exists
$\varepsilon>0$ such that $S=\phi([\pi/2-\epsilon, \pi/2+\epsilon])$
is of type $2$. (A computation reveals that $\mathcal{Z}(f_a)$ is
actually of type $3$ at $(1,1)$.)  By virtue of Theorem
\ref{t-oscillatory}, we now have $|\hat{\mu}(k,l)|\le C
(k^2+l^2)^{-1/4}$ for measures supported on the zero set in the torus,
so that the right hand side of \eqref{logenergycoeffs} converges for
$\alpha>1/2$, as in the proof of Theorem \ref{t-result}.  In
conclusion, $f_a$ is not cyclic in $\mathfrak{D}_\alpha$ for all
$\alpha>1/2$.

To address the remaining range $\alpha\in (0,1/2)$, we recall that $f(z_1,z_2)=1+z_1z_2$ was proved to be cyclic in $\mathfrak{D}_{\alpha}$ when $\alpha\leq 1/2$ using restriction arguments in \cite{BCLSS13II}. We now
note that
\[1-az_1-\bar{a} z_2+z_1z_2=(1-az_1)\cdot f\circ m_{a,0} (z_1,z_2),\]
where $m_{a,0}$ denotes a M\"obius automorphism of the bidisk of the
form \eqref{mobmap}. This means that $f_a$ is cyclic if $f\circ
m_{a,0}$ is. But since the $\mathfrak{D}_{\alpha}$-norm of a function
precomposed with a M\"obius map is comparable to the norm of the
function itself (see Section \ref{funcspaceprops}), the cyclicity of $f$ implies that $f_a$ is cyclic in $\mathfrak{D}_{\alpha}$ for $\alpha\leq 1/2$. This is because the approximating polynomials for $f$ transform into rational functions with no poles in the closed bidisk, and such functions can be approximated in multiplier norm by polynomials. An analogous argument shows that $f_a$ is cyclic for $\alpha\leq 1/2$; this of course is in agreement with the result in Section \ref{smallalpha}.

In the case where $a=0$ or $a=1$, the zero set can be viewed as straight lines in the torus: a local parametric
representation can be given as
\[t\mapsto t(p,q), \quad t\in [0,2\pi),\]
where $p,q\in \mathbb{Q}$.
(We have $(p,q)=(1,-1)$ in the case $a=0$, while $(p,q)=(1,0)$ and $(0,1)$ for $a=1$.)
In particular, both $\mathcal{Z}(f_0)$ and $\mathcal{Z}(f_1)$
are of infinite type. In this case, we consider the measures
induced by the integration current associated with $\mathcal{Z}(f_a)\cap \mathbb{T}^2$, as in \cite[Section 4.2]{BCLSS13II}, and note that their Fourier coefficients are constant along indices $(k,l)$ that satisfy
\[(k,l)\cdot (p,q)=kp+lq=0,\]
and equal to zero otherwise. Depending on whether both $p$ and $q$ are
non-zero or not, the series in \eqref{logenergycoeffs} will then
converge for $\alpha>1/2$ or $\alpha>1$. While the local curvatures of
the zero sets are the same in these cases, the cyclicity properties of
the polynomials differ because the norm in $\mathfrak{D}_{\alpha}$ is
of product type, and this is reflected in the fact that measures
supported on lines aligned along the coordinate axes are assigned
greater energy. $\hfill\blacklozenge$ 
\end{ex} 
Returning to polynomials, we note that a component of $\mathcal{Z}(f)\cap \mathbb{T}^2$ aligns with one of the coordinate axes precisely when $f$ contains a factor that is a polynomial in one variable only. If $f$ is a polynomial in one variable,
then $f$ is cyclic in $\mathfrak{D}_{\alpha}$ for all $\alpha \leq 1$.

Bearing this in mind, we formulate our main result for polynomials that vanish along curves.
\begin{thm}\label{bigalphathm}
Let $f\in \mathbb{C}[z_1,z_2]$ have no zeros in $\mathbb{D}^2$, and suppose $\mathcal{Z}(f)$ meets $\mathbb{T}^2$ along a curve. If $f$ is not a polynomial in one variable only, then $f$ is not cyclic in
$\mathfrak{D}_{\alpha}$ for $\alpha>1/2$.
\end{thm}
\begin{proof}
  Suppose $f$ is not a polynomial in one of the variables only. If
  $\mathcal{Z}(f)\cap\mathbb{T}^2$ contains a subset of type $2$, then
  $f$ is not cyclic by Theorem \ref{t-result}. If $\ZC(f)\cap \TT^2$
  contains no subset of type $2$, then we can find a M\"obius
  transformation $m_{a,b}$ that maps a curve contained in
  $\mathcal{Z}(f)\cap\mathbb{T}^2$ to a curve of type $2$.  To see
  this we note that since the zero set of $f$ cannot be aligned along
  the coordinate axes in $\TT^2$, $\ZC(f)\cap \TT^2$ must contain a
  curve with a tangent vector which is neither horizontal nor
  vertical.  Then, it is possible to parametrize a piece of
  $\ZC(f)\cap \TT^2$ via $\phi(t) = (t,\psi(t))$ where $\psi'(0)\ne
  0$.  In Section \ref{ss-oscillatory} we explained how to
  apply a M\"obius transformation to force the curve to be type 2 at a
  point (and hence in a neighborhood of the point on the curve).

By Theorem \ref{t-result} again, the function $F=f\circ m_{a,b}^{-1}$ is not cyclic for $\alpha>1/2$.
But this implies the non-cyclicity of $f$. For if there existed a sequence of polynomials $(p_n)_{n=1}^{\infty}$ such that
\[\|p_nf-1\|_{\alpha} \to 0 \quad n\to \infty, \]
then by precomposing with the M\"obius map $m_{a,b}^{-1}$, we would obtain
\[\|p_{n}\circ m^{-1}_{a,b}\cdot F-1\|_{\alpha}\leq C \|p_nf-1\|_{\alpha}\to 0,\]
contradicting the non-cyclicity of $F$ (as can be seen by approximating the sequence $p_n\circ m_{a,b}^{-1}$ in multiplier norm by polynomials).

It follows that $f$ is not cyclic, as claimed.
\end{proof}
To complete the capacity picture, we point out that the cyclicity result in Section \ref{smallalpha} shows that the $\alpha$-capacities of
$\mathcal{Z}(f)\cap \TT^2$ for any polynomial $f$ that does not 
vanish in $\mathbb{D}^2$ are equal to $0$ for all $\alpha\leq 1/2$.

\end{document}